\documentclass[reqno]{amsart}
\usepackage{amsmath,amsthm,amssymb}
\usepackage{graphicx,color}
\usepackage{hyperref}
\usepackage{mathrsfs}
\usepackage[utf8]{inputenc}
\usepackage{wasysym}
\usepackage{dsfont}
\usepackage{tikz-cd}
\usepackage[all]{xy}

\definecolor{Red}{cmyk}{0,1,1,0}

\definecolor{Blue}{cmyk}{1,1,0,0}

\makeatletter
\newtheorem*{rep@theorem}{\rep@title}
\newcommand{\newreptheorem}[2]{%
	\newenvironment{rep#1}[1]{%
		\def\rep@title{#2 \ref{##1}}%
		\begin{rep@theorem}}%
		{\end{rep@theorem}}}
\makeatother

\theoremstyle{plain}
\newtheorem{theorem}{Theorem}[section]
\newtheorem{corollary}[theorem]{Corollary}
\newtheorem{proposition}[theorem]{Proposition}

\theoremstyle{definition}
\newtheorem{definition}[theorem]{Definition}
\newtheorem{remark}[theorem]{Remark}

\newreptheorem{theorem}{Theorem}
\newreptheorem{lemma}{Lemma}
\newreptheorem{proposition}{Proposition}
\newreptheorem{corollary}{Corollary}

\usepackage[textsize=scriptsize,color=green!20,shadow]{todonotes}
\presetkeys{todonotes}{fancyline}{}

\makeatletter
\@namedef{subjclassname@2020}{\textup{2020} Mathematics Subject Classification}
\makeatother
\title{The Double Transpose of the Ruelle Operator}
\author[L. Cioletti]{L. Cioletti}
\address{Department of Mathematics, Universidade de Bras\'ilia, 70910-900, Bras\'ilia, Brazil}
\email{cioletti@mat.unb.br}
\author[A. van Enter]{A. van Enter}
\address{Bernoulli Instituut, Rijksuniversiteit Groningen, Nijenborgh 9, Groningen, 9747AG, The Netherlands}
\email{aenter@phys.rug.nl}
\author[R. Ruviaro]{R. Ruviaro}
\address{Department of Mathematics, Universidade de Bras\'ilia, 70910-900, Bras\'ilia, Brazil}
\email{ruviaro@mat.unb.br}
\thanks{This work was partially supported by the Coordena\c c\~ao de
	Aperfei\c coamento de Pessoal de N\'ivel Superior - Brasil (CAPES) - Finance Code 001.
	L. Cioletti and R. Ruviaro acknowledge financial support by FAP-DF. L. Cioletti
	is supported by CNPq through project PQ 313217/2018-1. }
\subjclass[2020]{Primary 37D35, 28Dxx; Secondary 60J05, 60F05}
\keywords{Thermodynamic Formalism, Ergodic Theory, Ruelle operator, Double Transpose, Eigenfunctions, Equilibrium states.}
\date{}

\allowdisplaybreaks

\begin{document}
\maketitle

\begin{abstract}
	In this paper we study the double transpose of the $L^1(X,\mathscr{B}(X),\nu)$-extensions of the Ruelle transfer 
	ope\-ra\-tor $\mathscr{L}_{f}$ associated to a general real continuous potential $f\in C(X)$,
	where $X=E^{\mathbb{N}}$, the alphabet $E$ is any compact metric space and $\nu$ is a maximal eigenmeasure. 
	For this operator, denoted by $\mathbb{L}^{**}_{f}$, we prove the existence of some non-negative eigenfunction, in the Banach lattice sense,
	associated to $\rho(\mathscr{L}_{f})$, the spectral radius of the Ruelle operator acting on $C(X)$. 
	As an application, we obtain a sufficient condition ensuring that 
	the natural extension of the Ruelle operator to 
	$L^1(X,\mathscr{B}(X),\nu)$ 
	has an eigenfunction associated to $\rho(\mathscr{L}_{f})$. 
	These eigenfunctions agree with the usual maximal eigenfunctions,  
	when the potential $f$  belongs to the  H\"older, Walters or Bowen class. 
	We also construct solutions to the classical 
	and generalized variational problem, 
	using the eigenvector constructed here.
\end{abstract}

\section{Introduction}

The Ruelle transfer operator, or simply the Ruelle operator, $\mathscr{L}_{f}, $ 
has its roots in the transfer matrix method introduced 
by Kramers and Wannier \cite{MR0004803} and (independently) by Montroll \cite{montroll1941}, 
to study the famous Ising model. 
This operator, acting on an infinite-dimensional vector space, 
was introduced in 1968 by David Ruelle 
\cite{MR0234697} to give a rigorous mathematical description of a relation between 
local and  global properties of a one-dimensional system 
composed of infinitely many particles, subject  
to an infinite-range potential. In particular, under an 
appropriate decay condition on the interaction (a local condition), 
uniqueness of the Gibbs measure (a global property) was proved. 
The Ruelle operator is one of the fundamental tools in 
Ergodic Theory/Thermodynamic Formalism, and one of the most important
results about this operator is the so-called Ruelle-Perron-Frobenius
theorem. Among other things,
the Ruelle-Perron-Frobenius theorem generalizes 
the classical Perron-Frobenius theorem for matrices 
to a class of positive operators acting on a suitable infinite-dimensional real vector space.

With the advent of the Markov partitions due to 
Adler and Weiss \cite{MR0212156} 
(for continuous ergodic automorphisms of the two-dimensional torus), 
Sinai \cite{MR0233038} (for Anosov diffeomorphisms)
and Bowen \cite{MR0277003} (general case),
remarkable applications of this operator to hyperbolic dynamical systems 
on compact manifolds were further obtained by Ruelle, Sinai and Bowen, 
see \cite{MR0387539,MR0234697,MR0399421}.
Since its creation, this operator has remained a major tool, which has had a great influence in many fields
of pure and applied mathematics. 
In particular, it has been  a powerful tool to study topological dynamics, 
invariant measures for Anosov flows, statistical mechanics in one dimension,
meromorphy of the Selberg and Ruelle dynamical zeta functions, multifractal analysis, 
Lyapunov exponents for product of random matrices, 
conformal dynamics in one dimension and fractal dimensions of horseshoes, just to name a few. Regarding these topics we refer the reader to  
\cite{MR3024807,MR556580,MR2423393,MR1143171,MR742227,MR1085356,MR2651384,MR1920859,MR0399421} 
and references therein.

The spectral analysis of these transfer operators is deeply connected 
with fundamental problems in Ergodic Theory,  
and Classical and Quantum Statistical Mechanics
on the one-dimensional lattice. For example, the maximal eigendata
(eigenvalue, eigenfunction, eigenmeasure and so on)
of the Ruelle transfer operator can be used to compute  
and determine uniqueness of the solutions 
of a central problem in Thermodynamic Formalism, 
introduced by Ruelle \cite{MR0217610} and Walters \cite{MR0390180},
which is  a variational problem of the following form:
\begin{align}\label{variational-problem}
	\sup_{\mu\in \mathscr{M}_{\sigma}(X)}
	\{h_{\mu}(\sigma) + \int_{X} f\, \text{d}\mu\},
\end{align}
where $h_{\mu}(\sigma)$ is the measure-theoretical  (or Kolmogorov-Sinai) entropy of $\mu$
and $\mathscr{M}_{\sigma}(X)$ is the
set of all $\sigma$-invariant Borel probability measures on $X$,
see also \cite{MR1793194,MR1618769,MR0404584,MR1085356,MR648108}.

Before presenting the precise definition of the Ruelle operator, 
we need to introduce some more notation.
Let $(E,d_{E})$ denote a general compact metric space which is sometimes 
called the state space, and fix a Borel probability measure 
$p:\mathscr{B}(E)\to [0,1]$ defined on $E$, having full support.
This condition will be important later, when we will talk about the extension of the Ruelle
operator to $L^1(\nu)\equiv L^1(X,\mathscr{B}(X),\nu)$.
We refer to this measure $p$ as the {\it a priori} measure. 
Consider the infinite product space (on the half-line) $X=E^{\mathbb{N}}$, 
endowed with any metric $d:X\times X\to [0,\infty)$ inducing the product topology,
and let $\sigma:X\to X$ be the left shift map. As usual we write 
$C(X)\equiv C(X,\mathbb{R})$ to 
denote the space of all real continuous
functions defined on $X$ and always assume that it is endowed with its standard
norm $\|\cdot\|_{\infty}$.  

Finally, for a fixed potential $f\in C(X)$, we define the 
Ruelle operator $\mathscr{L}_{f}:C(X)\to C(X)$ 
as being the linear operator that sends a continuous function 
$\varphi$ to another continuous function $\mathscr{L}_{f}\varphi$, 
which is given by the following expression 
\begin{align}\label{def-operador-ruelle}
	\mathscr{L}_{f}\varphi(x) = \int_{E} \exp(f(ax))\varphi(ax)\, \text{d}p(a),
	\qquad
	\text{where}\ \ ax\equiv (a,x_1,x_2,\ldots).
\end{align}
The spectral radius of $\mathscr{L}_{f}$ acting
on $(C(X),\|\cdot\|_{\infty})$ is denoted by $\rho(\mathscr{L}_{f})$.

In the sequel we summarize some of the classical results about
the Ruelle operator and its maximal eigendata in both finite and 
compact state space cases.  

Before we  proceed, we present some of the definitions of the most used regularity conditions in Thermodynamic Formalism. 
Recall that a function $\omega:[0,+\infty)\to [0,+\infty)$ is called 
a modulus of continuity, if $\omega$ is a continuous, increasing and concave function, such that
$\omega(0)=0$. 

The $n$-th variation of a function $f:X\to\mathbb{R}$ is defined by
\[
\mathrm{var}_{n}(f) \equiv \sup\{|f(x)-f(y)|: \ x_i=y_i,\ i=1,\ldots,n \}.
\] 
For the Birkhoff sum  we write $S_n(f) = f +f \circ \sigma +.... +f \circ \sigma^n $.

We say that a function $f$ satisfies: 
\begin{enumerate}
	\item the generalized H\"older condition,  if there is a constant $C>0$ and a modulus of continuity $\omega$ such  $|f(x)-f(y)|\leq C \omega(d(x,y))$, for all $x,y\in X$;
	\item the condition of summable variations,  if 
	$\displaystyle\sum_{n\in\mathbb{N}} \mathrm{var}_{n}(f)<+\infty$;
	
	\item Walters' condition, if 
	$\displaystyle\lim_{k\to\infty} \sup_{n\in\mathbb{N}} \mathrm{var}_{n+k}(S_n(f))=0$;
	
	\item Bowen's condition, if there is $k\in\mathbb{N}$ such that 
	$\displaystyle\sup_{n\in\mathbb{N}} \mathrm{var}_{n+k}(S_n(f))<+\infty$.
\end{enumerate}

The problems of finding the maximal positive eigenvalue for $\mathscr{L}_{f}$ 
and its respective eigenfunction and eigenmeasure are already solved if the potential $f$
satisfies one of the four conditions presented above.
For finite state spaces see, for example, 
\cite{MR1793194,MR1841880,MR2423393,MR1618769,MR1085356,MR0234697,MR2129258,MR0466493,MR1783787,MR2342978} 
and for the case of general compact metric state space see, for example, 
the references \cite{MR2864625,CLS20,MR3538412,MR4059795,MR3377291,MR576928}

\bigskip 

Historically, the results on the existence of  eigenfunctions
for the Ruelle operator started with the investigation of Lipschitz and 
H\"older potentials defined over symbol spaces with finite-state space,
see \cite{MR1793194,MR2423393,MR0234697,MR2129258}.
The literature about the Ruelle operator associated to such potentials 
is vast, and it lies at the heart of the most important applications of the 
Ruelle operator in several branches of  pure and applied mathematics.
We also mention here that the investigation of the basic
properties of this operator 
for H\"older potentials is a very important chapter in the 
theory of Thermodynamic Formalism.
In the nineties and the beginning of the two-thousands the 
interest in this operator, defined on shifts with infinite alphabets, increased, motivated 
in part by applications to non-uniformly
hyperbolic dynamical systems, see 
\cite{MR1107025,MR1639451,MR1853808,MR1738951,MR1955261,MR1468111}
and references therein. Simultaneously, investigations about 
this operator associated to potentials 
belonging to more general function spaces, like Walters and 
Bowen spaces, were carried out. 
Nowadays we can say that we have a 
%rather complete 
well-developed theory on this subject, 
but in both cases of finite-state space and infinite-state space, 
the problem of determining necessary and sufficient conditions on the 
potential ensuring the existence of positive eigenfunctions for the Ruelle 
operator associated to its spectral radius remains open. 

\bigskip 

When considering potentials $f\in C(X)$ 
living outside the Bowen space, even for finite-state spaces, 
in very few cases the maximal eigendata of the Ruelle
operator can be obtained. The obstacle one faces to perform
the spectral analysis of the operator in this 
case, by using current mathematical technology, 
is not only of a technical nature.
For such potentials a phase transition,   
($|\mathscr{G}^{*}(f)|>1$, see Proposition \ref{prop-exist-auto-medida} for the definition of $\mathscr{G}^{*}(f)$), 
can take place and this introduces a bunch of new difficulties. 
The classical example, in the Thermodynamic Formalism, 
of a continuous potential  where such phenomena can occur is given by 
a Hofbauer-type potential see \cite{MR0435352}. 
Basically, in such generality no general mathematical 
theory exists, and as far as we know, things can only be  
handled on a case-by-case basis.    
Other one-dimensional examples of systems which have phase transitions 
are the Dyson models, 
the Bramson-Kalikow and Berger-Hoffman-Sidoravicius $g$-measure examples, 
and the Fisher-Felderhof renewal-type examples,
see \cite{bhs2017,MR1244665,MR3350377,MR0436850,FF70,MR3928619}.

\medskip

The goal of this paper is to initiate the study of 
the double transpose of some suitable extensions $\mathbb{L}_{f}:L^1(\nu)\to L^1(\nu)$,
where $\nu\in\mathscr{G}^{*}(f)$,  
of the Ruelle transfer operator 
$\mathscr{L}_{f}$, associated to a general real continuous potential $f:X\to\mathbb{R}$, 
and the main results obtained here are
Theorems \ref{teo-autofunc-bidual} and \ref{teo-Rad-Nik-autofuncao}.

\medskip 

The approach taken here is novel, in so far as it focuses on the bidual of either 
$C(X)$ or $L^{1}(\nu)$
and the extension of the Ruelle operator to these spaces. 
It allows us to find a new sufficient condition to 
solve the problem of existence of maximal 
eigenfunctions for continuous potentials having low regularity properties. 

\medskip

As mentioned above, one of the main ideas of this paper 
is to carefully study the double transpose $\mathbb{L}_{f}^{**}$, 
where $\mathbb{L}_{f}:L^1(\nu)\to L^1(\nu)$ is an extension   
of the classical Ruelle operator $\mathscr{L}_{f}$. 
We prove that $\rho(\mathscr{L}_{f})$ 
is an eigenvalue of $\mathbb{L}_{f}^{**}$ 
and has associated to it a non-negative (in the Banach lattice sense) 
eigenvector $\xi_{f}$, for any continuous potential $f$ in $C(X)$. 
Afterward, we show how we can 
use this eigenvector $\xi_{f}$ to construct equilibrium states, under mild conditions.

\medskip

When working with potentials with very low regularity properties,  
to prove the existence of non-negative eigenfunctions for 
$\mathbb{L}_{f}$ is in general a hard task, 
see \cite{JLMS:JLMS12031,MR3350377,MR0435352}.

Here we work on product spaces of the form $X=E^{\mathbb{N}}$,
nevertheless the problems considered here 
are of more general interest to the Thermodynamic Formalism on various sorts
of compact subspaces as well. 
We restrict ourselves to product spaces for the sake of simplicity.
The statements of our main theorems and its respective proofs 
are easily adapted to compact subshifts, 
since the techniques employed here are based on the general 
theory of linear operators. We shall remark that our approach
does not need the expansivity hypothesis, 
sometimes considered in the literature on Thermodynamic Formalism.

\subsection{Directions for future research}

One of our main results, Theorem \ref{teo-medida-gibbs-invariante}, 
states that for any continuous potential $f$, and $\nu\in\mathscr{G}^{*}(f)$, there is a positive 
element $\xi_{f}\in L^{1}(\nu)^{**}$ which determines a $\sigma$-invariant Borel finitely 
additive measure given by 
$A\longmapsto \xi_{f}(\ell_{A})\equiv\mu(A)$, where $\ell_{A}$ denotes the bounded
linear functional on $L^1(\nu)$ given by $\ell_{A}(\varphi)=\int_{X}1_{A}\varphi\, d\nu$.

By Theorem 3.1 in \cite{To20} we have 
\[
L^{1}(\nu)^{**}
\cong
\left\{
\gamma:\mathscr{B}(X)\to\mathbb{R}: 
\gamma\in \mathrm{ba}(\mathscr{B}(X))\ \text{satifying} \ \nu(Z)\!=\!0\! \Longrightarrow \gamma(Z)=0
\right\},  
\]
where $\mathrm{ba}(\mathscr{B}(X))$ stands for the space of all Borel bounded finitely additive
signed measures. By Theorem 3.1 in \cite{To20}, we can identify the finitely additive measure $\mu$, constructed above, 
with a bounded linear functional $F_{\mu}:L^{\infty}(\nu)\to \mathbb{R}$. The restriction of $F_{\mu}$ to $C(X)$
determines a countably additive $\sigma$-invariant measure, because of the Riesz-Markov Theorem.
It would be very interesting to know whether such a Borel 
probability measure is actually an equilibrium state associated to $f$,
and how to determine the support of  such a measure in 
terms of its Yosida-Hewitt decomposition $\mu=\mu_{a}+\mu_{c}$,
see \cite{To20,MR0045194} for details on this decomposition.

Another interesting question related to Theorem \ref{teo-medida-gibbs-invariante} is the following.
Consider $\nu$ the barycenter of the convex set $\mathscr{G}^{*}(f)$ and let 
$\mathscr{H} \equiv \{ \xi\in L^1(\nu)^{**} :\ \mathbb{L}_{f}^{**}\xi\!=\!\rho(\mathscr{L}_{f})\xi \}$.
Some of the results in \cite{CMRS-2020} suggest that $\mathrm{dim}_{\mathbb{R}}\, \mathscr{H}\!=\! \#\mathrm{ex}(\mathscr{G}^{*}(f))$ (the set of extreme points in $\mathscr{G}^{*}(f)$). 
If such relation indeed holds, it has as a consequence 
a new criterion for the existence of phase transitions for one-dimensional one-sided lattice
statistical mechanical spin systems, which in turn would reduce this problem to a problem of 
spectral analysis of Markov operators. 

Yet another question, which seems to be more puzzling, is the following. When 
the Yosida-Hewitt decomposition of $\mu$ is trivial, meaning that $\mu=\mu_{a}$
is purely finitely additive, can any statistics along typical orbits of observables 
be obtained, similarly to the countably additive setting? Are there any relations 
between  this decomposition and the existence of sigma-finite shift-invariant measures which 
are not finite?

The connection between $\mathrm{Eq}(f)$, the set of equilibrium states associated to $f$,
and $\mathscr{G}^{*}(f)$ is still mysterious when $f$ has low regularity properties. 
For finite alphabets it is very well known that any probability measure $\nu\in \mathscr{G}^{*}(f)$
is fully supported and the same holds in the metric compact case, see \cite{CMRS-2020}. 
On the other hand, Jenkinson, and Israel and Phelps showed in \cite{MR753071,MR2226487} 
that for every non-empty collection $\mathscr{C}$ of ergodic measures 
which is a closed subset of $\mathscr{M}_{\sigma}(X)$ (in the weak-$*$-topology), 
there is a continuous potential $f$ such that the closure of the convex hull of $\mathscr{C}$
is precisely the set of the equilibrium states for $f$. In particular, if $x\in X$ is such that $\sigma(x)=x$
and $\delta_{x}$ denotes the Dirac delta measure concentrated in $x$, then there is a continuous potential
$f$ such that $\mathrm{Eq}(f)=\{\delta_{x}\}$. On the other hand, any $\nu\in \mathscr{G}^{*}(f)$
is such that $\mathrm{supp}(\nu)=X$ and so $\nu\perp \delta_{x}$ and therefore they are not linked, as in the classical 
cases (H\"older, Walters and Bowen), by an eigenfunction. Could they be linked by the generalized 
eigenfunctions $\xi\in \mathscr{H}$? The example in the last section suggests that the answer to this 
question is affirmative, at least in the particular cases where the set of equilibrium states is a singleton
and the link is provided by the restriction of $F_{\mu}$ as described above. 

\subsection{Organization of the paper}
The paper is organized as follows. 
In Section \ref{sec-pressao} we study the pressure functional for 
general continuous potentials defined on $E^{\mathbb{N}}$, where 
$E$ is any compact metric space. In such a general setting we prove 
the existence of an eigenmeasure associated to $\rho(\mathscr{L}_{f})$, the spectral radius 
of the Ruelle operator acting on $C(X)$. 
In particular, we prove that the set 
$\mathscr{G}^{*}(f)\equiv \{\nu\in\mathscr{M}_{1}(X): 
\mathscr{L}_{f}^{*}\nu=\rho(\mathscr{L}_{f})\nu \}$ is always non-empty, for any continuous potential $f$.
In  Section \ref{sec-duplo-adjunto} 
we study the double transpose of the Ruelle operator, associated to 
a general continuous potential, acting on both the spaces $C(X)$ 
and $L^1(\nu)$, where $\nu\in \mathscr{G}^{*}(f)$. 
For any continuous potential $f$, the existence of a 
positive eigenvector $\xi_{f}$ for the double transpose of the 
Ruelle operator acting on $L^1(\nu)$, associated to $\rho(\mathscr{L}_{f})$,
is established. 
In Section \ref{sec-inv-measures} 
we show how to construct for each eigenmeasure $\nu \in \mathscr{G}^{*}(f)$ a 
shift-invariant finitely additive measure $\mu$ by using the eigenvector $\xi_{f}$. 
When this finitely additive measure is a probability  
measure (countably additive) we show that $\mu\ll \nu$, and consequently
how to construct an eigenfunction for the natural 
extension of the Ruelle operator to  $L^1(\nu)$, associated to $\rho(\mathscr{L}_{f})$. 
In Section \ref{sec-Man-Pomeau}, we present an example where the shift-invariant measure induced by 
the eigenvector $\xi_{f}$ is a finitely additive but not countably additive measure.
In  Section \ref{sec-prob-variacional} we study a generalization 
of the classical variational problem \eqref{variational-problem} 
for continuous potentials defined on $E^{\mathbb{N}}$, where $E$
is any compact metric space. 
We introduce a generalization of the Kolmogorov-Sinai entropy 
in order to obtain a meaningful definition of entropy for 
shifts in symbol spaces having uncountable alphabets. 
The problem \eqref{variational-problem} is then reformulated, 
in a natural way, in this setting and next we show how to solve 
this variational problem using the shift-invariant probability measures 
constructed in Section \ref{sec-inv-measures}.
In Section \ref{sec-conc-remarks}, we provide some concluding remarks.

\section{Pressure Functional and Eigenmeasures}
\label{sec-pressao}

In this section we prove the existence of an eigenmeasure for the transpose of the 
Ruelle operator, associated to $\rho(\mathscr{L}_{f})$.
To this end we recall that $C(X)^*$ is isometrically isomorphic 
to $\mathscr{M}_s(X)$, so we can think of
$\mathscr{L}_{f}^{*}$ as a bounded linear operator 
acting on $\mathscr{M}_s(X)$.
As a consequence of the Riesz-Markov Theorem, it is easy to see that 
the transposed operator sends $\nu\mapsto \mathscr{L}_{f}^*\nu$, where
$\mathscr{L}_{f}^*\nu$ is the unique signed Radon measure satisfying
\[
\int_{X} \mathscr{L}_{f}\varphi \, \text{d}\nu
=
\int_{X} \varphi \, \text{d}[\mathscr{L}_{f}^*\nu],
\qquad \forall \varphi\in C(X).
\]

\begin{theorem}[Pressure Functional]\label{theo21}
	Let $X=E^{\mathbb{N}}$, where $E$ is a general compact metric space
	and $f\in C(X)$ a continuous potential.  
	Then there is a real number $P(f)$, called the pressure of the potential $f$, such that 
	\begin{align}\label{teo-existencia-pressao}
		\lim_{n\to\infty}\,
		\sup_{x\in X}
		\left| 
		\frac{1}{n} \log \mathscr{L}_{f}^n(1)(x) -P(f)
		\right|
		= 
		0.
	\end{align}
\end{theorem}

\begin{proof}
	For finite-state space and H\"older potentials the proof can be found in 
	\cite{MR1793194,MR2423393,MR1085356,MR2129258}.
	For general compact metric spaces and continuous potentials, see 
	\cite{CLS20,MR1143171,2017arXiv170709072S}. 
\end{proof}

\begin{remark}
	Note that the above result does not contradict the existence of metastable 
	states proved by Sewell \cite{MR0468999} since $n^{-1}\log \mathscr{L}_{f}^n(1)(x)$,
	for each $f\in C(X)$, can be written as a finite-volume pressure 
	of a translation invariant interaction $\Phi=(\Phi_{\Lambda})_{\Lambda\Subset \mathbb{N}}$ 
	on the lattice $\mathbb{N}$, satisfying the following regularity condition:
	$\sum_{\Lambda\ni 1}\|\Phi_{\Lambda}\|_{\infty}<+\infty$.
\end{remark}

We remark that a one-sided ($\mathbb{N}$) continuous potential appears to 
provide substantially stronger regularity conditions 
on the associated measures than continuous potential in 
the two-sided ($\mathbb{Z}$) case does.

\bigskip

The next result relates the pressure functional to the logarithm of the
spectral radius for continuous potentials defined over general metric compact symbolic spaces.

\begin{corollary}
	For any $f\in C(X)$ we have $P(f)=\log \rho(\mathscr{L}_{f})$.
\end{corollary}

\begin{proof}
	The idea is to use Gelfand's Formula for the spectral radius.
	Since $\mathscr{L}_{f}$ is a positive operator,
	$f$ is continuous and $X$ is compact, 
	for each $n\in\mathbb{N}$ we can ensure the existence of 
	some $x_n\in X$ for which 
	\begin{align*}
		\|\mathscr{L}_{f}^{n}\|_{\mathrm{op}}
		\equiv
		\sup_{\|\varphi\|_{\infty}=1} \|\mathscr{L}_{f}^n(\varphi)\|_{\infty}
		=
		\|\mathscr{L}_{f}^n(1)\|_{\infty}
		=
		\mathscr{L}_{f}^n(1)(x_n)
		\geq \exp(-n\|f\|_{\infty}).
	\end{align*}
	The above inequality implies that the spectral radius is non-null. 
	By taking the $n$-th root of both sides above 
	and then the logarithm, we get from Theorem \ref{theo21} 
	that 
	\[
	\log
	\|\mathscr{L}_{f}^{n}\|_{\mathrm{op}}^{\frac{1}{n}}
	=
	\frac{1}{n}\log \mathscr{L}_{f}^n(1)(x_n)
	\xrightarrow{\ n\to\infty\ }
	P(f)
	\geq 
	-\|f\|_{\infty}.
	\] 
	On the other hand, from the boundedness of  $\mathscr{L}_{f}$ and Gelfand's formula, it
	follows that the limit, when $n$ goes to infinity, of the lhs above 
	is precisely the logarithm  of the spectral radius of $\mathscr{L}_f$.
	Therefore $\log\rho(\mathscr{L}_{f})=P(f)$.
\end{proof}

Now we are ready to prove the main result of this section.

\begin{proposition}\label{prop-exist-auto-medida} 
	For any continuous potential $f$ we have
	\[
	\mathscr{G}^{*}(f) 
	\equiv 
	\{
	\nu \in \mathscr{M}_1(X): \mathscr{L}_{f}^*\nu =\rho(\mathscr{L}_{f})\nu 
	\}
	\]
	is non-empty.
\end{proposition}

\begin{proof}
	Note that the mapping
	$\mathscr{M}_1(X)\ni 
	\gamma
	\mapsto 
	(\mathscr{L}_{f}^{*}(\gamma)(X))^{-1}\, \mathscr{L}_{f}^{*}(\gamma) 
	$
	sends $\mathscr{M}_1(X)$ to itself. Since this set is also convex 
	and compact, in the weak-$*$-topology which is Hausdorff when $X$ is metric and compact, 
	it follows from the continuity of $\mathscr{L}_{f}^{*}$ and the Tychonov-Schauder Theorem
	that this mapping has at least one fixed point $\nu$. 
	Note that this fixed point is an 
	eigenmeasure for the transpose of the Ruelle operator, i.e., 
	$\mathscr{L}_{f}^{*}(\nu)=(\mathscr{L}_{f}^{*}(\nu)(X))\, \nu$.
	
	We claim that the following bounds, 
	$$
	\exp(-\|f\|_{\infty})
	\leq 
	\mathscr{L}_{f}^{*}(\nu)(X) 
	\leq 
	\exp(\|f\|_{\infty}),
	$$ 
	hold for any $\nu\in\mathscr{M}_1(X)$, satisfying $\mathscr{L}_{f}^{*}(\nu)=(\mathscr{L}_{f}^{*}(\nu)(X))\, \nu$. 
	In fact, the lower bound holds for any Borel probability measure over
	$X$, since 
	\begin{align*}
		\exp(-\|f\|_{\infty})\leq 
		\int_{X} \mathscr{L}_{f}(1) \, \text{d}\gamma 
		=
		\int_{X} 1 \, \text{d}[\mathscr{L}_{f}^{*}\gamma]
		=
		\mathscr{L}_{f}^{*}(\gamma)(X),
		\qquad \forall \ \gamma\in \mathscr{M}_1(X).
	\end{align*}
	
	For the upper bound it is necessary to use that $\nu$ is an eigenmeasure
	associated to the eigenvalue $\mathscr{L}_{f}^{*}(\nu)(X)$
	and the argument is as follows:
	\begin{align*}
		\exp(\|f\|_{\infty})
		\geq 
		\int_{X} \mathscr{L}_{f}(1) \, \text{d}\nu 
		=
		\int_{X} 1 \, \text{d}[\mathscr{L}_{f}^{*}\nu]
		=
		\mathscr{L}_{f}^{*}(\nu)(X)
		\int_{X} 1 \, \text{d}\nu
		=
		\mathscr{L}_{f}^{*}(\nu)(X).
	\end{align*}
	
	From the claim it follows that 
	\[
	\overline{\rho(\mathscr{L}_{f})} 
	\equiv 
	\sup\{ \mathscr{L}_{f}^{*}(\nu)(X): 
	\mathscr{L}_{f}^{*}(\nu)
	=
	(\mathscr{L}_{f}^{*}(\nu)(X))\, \nu  
	\}
	<
	+\infty.
	\]
	A simple compactness argument shows that there is $\nu\in \mathscr{M}_{1}(X)$ so that 
	$\mathscr{L}_{f}^{*}\nu=\overline{\rho(\mathscr{L}_{f})}\nu$. 
	Indeed, let $(\nu_n)_{n\in\mathbb{N}}$ 
	be a sequence such that 
	$\mathscr{L}_{f}^{*}(\nu_n)(X)\uparrow \overline{\rho(\mathscr{L}_{f})}$, 
	when $n$ goes to infinity. 
	Since $\mathscr{M}_1(X)$ is a compact metric space
	in the weak-$*$-topology, we can assume, up to subsequence convergence, that $\nu_n\rightharpoonup \nu$. This convergence,
	together with the continuity of $\mathscr{L}_{f}^{*}$, provides 
	\[
	\mathscr{L}_{f}^{*}\nu 
	= 
	\lim_{n\to\infty}\mathscr{L}_{f}^{*}\nu_n 
	=
	\lim_{n\to\infty}\mathscr{L}_{f}^{*}(\nu_n)(X)\nu_n
	=
	\overline{\rho(\mathscr{L}_{f})}\, \nu.
	\]
	Therefore the set 
	$
	\{
	\nu \in \mathscr{M}_1(X): 
	\mathscr{L}_{f}^*\nu =\overline{\rho(\mathscr{L}_{f})}\, \nu 
	\}
	\neq 
	\emptyset
	$.
	It remains to show that $\overline{\rho(\mathscr{L}_{f})}$ 
	is the spectral radius of $\mathscr{L}_{f}$.
	In order to prove this statement we first observe that 
	\begin{align*}
		\overline{\rho(\mathscr{L}_{f})}^{\, n}
		=
		\int_{X} \mathscr{L}_f^{n}(1)(x)\, \text{d}\nu 
		\leq 
		\|\mathscr{L}_{f}^{n}\|_{\mathrm{op}}.
	\end{align*}
	and therefore $\overline{\rho(\mathscr{L}_{f})}\leq \rho(\mathscr{L}_{f})$.
	From the uniform convergence provided by  
	Theorem \ref{teo-existencia-pressao} and  Jensen's inequality
	we get
	\begin{align*}
		\log\rho(\mathscr{L}_{f})
		=
		\lim_{n\to\infty} \frac{1}{n} \int_{X} \log\mathscr{L}_{f}^{n}(1)\, \text{d}\nu
		\leq
		\lim_{n\to\infty} \frac{1}{n} \log\int_{X} \mathscr{L}_{f}^{n}(1)\, \text{d}\nu
		=
		\log \overline{\rho(\mathscr{L}_{f})},
	\end{align*}
	thus proving that $\overline{\rho(\mathscr{L}_{f})}=\rho(\mathscr{L}_{f})$.
\end{proof}

\section{The Double Transpose and Its Eigenfunctions}
\label{sec-duplo-adjunto}

In this section we establish some elementary properties of the 
double transpose of the Ruelle operator associated to a 
continuous potential $f$.

By using the isomorphism $C^{**}(X) \backsimeq \mathscr{M}_s(X)^{*} $, 
we can consider the double transpose of $\mathscr{L}_{f}:C(X)\to C(X)$ 
as the unique linear operator 
$\mathscr{L}_{f}^{**}:\mathscr{M}_s(X)^{*}\to \mathscr{M}_s(X)^{*}$ 
sending $\xi \mapsto \mathscr{L}_{f}^{**}( \xi) $,
defined for each $\mu\in \mathscr{M}_s(X)$ by
\[
\mathscr{L}_{f}^{**}( \xi)(\mu)  \equiv  \xi(\mathscr{L}_{f}^{*}\mu). 
\]

By identifying $C(X)$ with the image of the natural map
$J:C(X)\to \mathscr{M}_s(X)^{*}$, defined by
$J(\varphi)(\mu)=\mu(\varphi)$, we can think 
of $\mathscr{L}_{f}^{**}$ as a bounded linear extension 
of the Ruelle operator
\begin{displaymath}
	\xymatrix{
		\mathscr{M}_s(X)^{*}  \ar[r]^{\mathscr{L}_{f}^{**}} 
		& \mathscr{M}_s(X)^{*} 
		\\
		\mathscr{M}_s(X) \ar[u]^{\text{duality}} 
		& \ar[l]_{\mathscr{L}_{f}^{*} }  \mathscr{M}_s(X)\ar[u]_{\text{duality}}
		\\
		C(X) \ar[u]^{\text{duality}}  \ar[r]^{\mathscr{L}_{f}} 
		&  C(X). \ar[u]_{\text{duality}}   }
\end{displaymath}

\medskip
Let $\mathscr{M}_{+}(X)$ denote the set of all finite positive 
Borel measures over $X$. We say that an element $\xi\in \mathscr{M}_s(X)^{*}$ 
is positive if $\xi(\mathscr{M}_{+}(X)) \subset [0,+\infty)$.

Note that, if $h_f$ is a positive continuous eigenfunction of $\mathscr{L}_{f}$
associated to the maximal eigenvalue $\mathscr{L}_{f}$, then 
for any finite signed measure $\mu$ we have 
\begin{align*}
	\mathscr{L}_{f}^{**}(J(h_f))(\mu)  
	&=
	J(h_f)(\mathscr{L}_{f}^{*}(\mu))
	\\
	&=
	\mathscr{L}_{f}^{*}(\mu)(h_f)
	\\
	&=
	\mu(\mathscr{L}_{f}(h_f))
	=
	\rho(\mathscr{L}_{f}) \mu(h_f)
	=
	\rho(\mathscr{L}_{f})J(h_{f})(\mu),
\end{align*}
moreover, if $\mu\in\mathscr{M}_{+}(X)$, then $J(h_f)(\mu) = \mu(h_f)>0$.
Therefore $J(h_f)$ is a positive eigenvector of $\mathscr{L}^{**}_{f}$
associated to $\rho(\mathscr{L}_{f})$.

Let us assume that there is an eigenfunction $\xi$, which might not be necessarily positive,
for $\mathscr{L}^{**}_{f}$, associated to $\rho(\mathscr{L}_{f})$. Then for any 
signed measure $\mu$ we have 
\begin{align*}
	0
	=
	\mathscr{L}_{f}^{**}(\xi)(\mu)-\rho(\mathscr{L}_{f})\xi(\mu)
	=
	\xi(\mathscr{L}_{f}^{*}(\mu))-\rho(\mathscr{L}_{f})\xi(\mu)
	=
	\xi((\mathscr{L}_{f}^{*}-\rho(\mathscr{L}_{f}))\mu).
\end{align*}
The above equation implies, when such an eigenvector exists,  
that the range of the operator $\mathscr{L}_{f}^{*}-\rho(\mathscr{L}_{f})$ is 
contained in the kernel of $\xi$, i.e., 
$\mathscr{R}(\mathscr{L}_{f}^{*}-\rho(\mathscr{L}_{f}))\subset \mathrm{ker}(\xi)$.
On the other hand, if for some $\nu\in \mathscr{G}^{*}(f)$ we have 
\[
\overline{\mathscr{R}(\mathscr{L}_{f}^{*}-\rho(\mathscr{L}_{f}))}
\cap
\langle \nu \rangle 
=
\{0\},
\qquad\text{where}\  \langle \nu \rangle\ \text{is the subspace generated by} \ \nu,
\]
then, as a consequence of the Hahn-Banach theorem, we can guarantee the existence of 
at least one continuous functional $\xi \in \mathscr{M}^{*}_{s}$ so that 
$\mathscr{R}(\mathscr{L}_{f}^{*}-\rho(\mathscr{L}_{f}))\subset \mathrm{ker}(\xi)$
and $\xi(\nu)=1$. Such a functional is clearly an eigenvector for 
$\mathscr{L}^{**}_{f}$ associated to $\rho(\mathscr{L}_{f})$. 
We  remark that the eigenvector $\xi$ is not necessarily 
positive, nor necessarily an element of $J(C(X))$, the image of the natural map.

As long as $\psi$ is a $\mathscr{B}(X)$-measurable real function satisfying 
$\|\psi\|_{\infty}<+\infty$, we can naturally define $\mathscr{L}_{f}\psi$,
since for every $x\in X$ the following integral is well defined and finite:
\[
\int_{E} \exp(f(ax))|\psi(ax)|\, \mathrm{d}p(a).
\]

\begin{definition}
	Let $(X,d_{X})$ a compact metric space and  $f: X\to\mathbb{R}$ be a function. 
	We say that $f$ is a {\bf Baire-class-one} function if there
	is a sequence of continuous functions converging to $f$ pointwise.
\end{definition}

\begin{proposition}
	Let $f$ be a continuous potential and suppose that there exists a 
	Baire-class-one real function $\psi:X\to\mathbb{R}$ satisfying 
	$0<m\leq \psi\leq M<+\infty$ and $\mathscr{L}_{f}\psi =\rho(\mathscr{L}_{f})\psi$.
	Then 
	\[
	\overline{\mathscr{R}(\mathscr{L}_{f}^{*}-\rho(\mathscr{L}_{f}))}
	\cap 
	\mathscr{M}_{1}(X)
	=
	\emptyset.
	\]
\end{proposition}

\begin{proof}
	The first step is to show that $(\mathscr{L}_{f}^{*}-\rho(\mathscr{L}_{f}))(\mu)(\psi)=0$,
	for any signed measure $\mu$. 
	Since we are not assuming that $\psi\in C(X)$, there is a small issue in
	using the duality relation for the Ruelle operator and its transpose. 
	But this issue can be easily overcome as follows. Let $(\psi_n)_{n\in\mathbb{N}}$
	a sequence of continuous functions pointwise converging to $\psi$. Then by
	the Dominated Convergence Theorem, for any finite signed measure $\mu$ we have 
	\begin{align*}
		(\mathscr{L}_{f}^{*}-\rho(\mathscr{L}_{f}))(\mu)(\psi)
		&=
		\lim_{n\to\infty}
		(\mathscr{L}_{f}^{*}-\rho(\mathscr{L}_{f}))(\mu)(\psi_n)
		\\[0.2cm]
		&=
		\lim_{n\to\infty}
		\mu(\mathscr{L}_{f}\psi_n)-\rho(\mathscr{L}_{f})\mu(\psi_n)
		\\[0.2cm]
		&=
		\mu(\mathscr{L}_{f}\psi)-\rho(\mathscr{L}_{f})\mu(\psi)
		\\[0.2cm]
		&=
		0.
	\end{align*}
	Suppose by contradiction that 
	$\mathscr{R}(\mathscr{L}_{f}^{*}-\rho(\mathscr{L}_{f}))\cap \mathscr{M}_1(X)\neq \emptyset$.
	Then there are $\mu\in\mathscr{M}_{s}(X)$ and 
	$\nu\in \mathscr{M}_{1}(X)$ such that $(\mathscr{L}_{f}^{*}-\rho(\mathscr{L}_{f}))\mu=\nu$. 
	By using the above equalities and the bounds on $\psi$, we get 
	$
	0<
	m
	\leq 
	\min(\psi)\nu(X)
	\leq 
	\nu(\psi)
	=
	(\mathscr{L}_{f}^{*}-\rho(\mathscr{L}_{f}))(\mu)(\psi)
	=
	0
	$
	and therefore $\mathscr{R}(\mathscr{L}_{f}^{*}-\rho(\mathscr{L}_{f}))\cap \mathscr{M}_1(X)=\emptyset$.
	
	We now prove that there is no sequence $(\mathscr{L}_{f}^{*}-\rho(\mathscr{L}_{f}))(\mu_n)$ 
	converging to a probability measure $\nu$ in the strong topology. 
	Suppose by contradiction that 
	\[
	\|(\mathscr{L}_{f}^{*}-\rho(\mathscr{L}_{f}))(\mu_n)-\nu\|_{T}\to 0,
	\] 
	when $n\to\infty$. 
	We have already shown that $(\mathscr{L}_{f}^{*}-\rho(\mathscr{L}_{f}))(\mu_n)(\psi)=0$, and 
	therefore we have
	\[
	m
	\leq 
	\nu(\psi)
	=
	|(\mathscr{L}_{f}^{*}-\rho(\mathscr{L}_{f}))(\mu_n)(\psi)-\nu(\psi)|
	\leq 
	\|(\mathscr{L}_{f}^{*}-\rho(\mathscr{L}_{f}))(\mu_n)-\nu\|_{T}\ 
	\|\psi\|_{\infty}.
	\]
	Since the rhs converges to zero we reach a contradiction. 
\end{proof}

\begin{theorem}
	Let $f$ be a continuous potential and suppose that there exists a 
	Baire-class-one real function $\psi:X\to\mathbb{R}$, satisfying 
	$0<m\leq \psi\leq M<+\infty$ and $\mathscr{L}_{f}\psi =\rho(\mathscr{L}_{f})\psi$.
	Then there is a positive element $\xi_{f}\in \mathscr{M}_{s}^{*}(X)$ 
	such that $\mathscr{L}_{f}^{**}\xi_{f}= \rho(\mathscr{L}_{f})\xi_{f}$.
\end{theorem}

\begin{proof}
	Let $(\psi_n)_{n\in\mathbb{N}}$ be a sequence in $C(X)$ such that 
	$\psi_n\to \psi$ pointwise. Since the sequence $\max\{ \min\{\psi_n,M \}, m\}$ 
	is continuous and converges pointwise to $\psi$ we can assume that $m\leq \psi_n\leq M$.
	Now we consider the sequence of linear functionals $(J(\psi_n))_{n\in\mathbb{N}}$ in the 
	bidual of $C(X)$. The natural map is an isometry and so $\|J(\psi_n)\|\leq M$.
	From the Banach-Alaoglu theorem it follows that the closed ball $\overline{B(0,M)}$ is 
	compact in the weak-$*$ topology, therefore  
	the sequence $(J(\psi_n))_{n\in\mathbb{N}}$ viewed as a topological net 
	has at least one convergent subnet $(J(\psi_{i(d)}))_{d\in D}$ so that 
	\[
	\lim_{d\in D} J(\psi_{i(d)}) = \xi_{f}.
	\]
	We claim that $\xi_{f}$ is a positive eigenfunction of $\mathscr{L}_{f}^{**}$.
	The positivity of $\xi_{f}$ is trivial, because for any $\mu\in \mathscr{M}_{+}(X)$ we have 
	$J(\psi_{i(d)})(\mu)\geq m$. In particular, $\xi_{f}$ is not the null vector. 
	To finish the proof it is enough to show that $\mathscr{R}(\mathscr{L}_{f}^{*}-\rho(\mathscr{L}_{f}))\subset \mathrm{ker}(\xi_{f})$.
	Indeed, for any $\mu\in\mathscr{M}_{s}^{*}(X)$ we have 	
	\begin{eqnarray*}
		\xi_{f} ((\mathscr{L}_{f}^{*}-\rho(\mathscr{L}_{f}))(\mu) )
		&=&
		\lim_{d\in D}J(\psi_{i(d)}) ((\mathscr{L}_{f}^{*}-\rho(\mathscr{L}_{f}))(\mu) )
		\\[0.2cm]
		&=&
		\lim_{d\in D} \mu (\mathscr{L}_{f}(\psi_{i(d)})-\rho(\mathscr{L}_{f})\psi_{i(d)} )
		\\[0.2cm]
		&=&
		\mu (\mathscr{L}_{f}(\psi)-\rho(\mathscr{L}_{f})\psi )
		\\[0.2cm]
		&=& 
		0.
	\end{eqnarray*} 
\end{proof}

In what follows we consider the extension 
$\mathbb{L}_{f}:L^1(\nu)\to L^1(\nu)$  
of the classical Ruelle operator $\mathscr{L}_{f}:C(X)\to C(X)$.
We say that a bounded positive linear operator
$\mathbb{L}_{f}:L^1(\nu)\to L^1(\nu)$ 
is an extension of the transfer operator
$\mathscr{L}_{f}:C(X)\to C(X)$ if the vector space
$C(X)$ embeds in $L^1(\nu)$ and for any $\varphi\in C(X)$ we have
$\mathbb{L}_{f}[\varphi]_{\nu}\cap C(X) = \{\mathscr{L}_{f}\varphi\}$.

When $E$ is a finite set, 
we can apply either Proposition 2.2 of \cite{MR1014246} or Corollary 4 in \cite{MR0466493} to 
obtain this $L^1(\nu)$-extension. 
For uncountable alphabets this extension is obtained in \cite{CMRS-2020}. 
The key point in obtaining the extension is to 
show that any $\nu\in\mathscr{G}^{*}(f)$ is fully supported, that is,  
$\mathrm{supp}(\nu)=X$. In \cite{CMRS-2020} the authors prove that 
whenever $\mathrm{supp}(p)=E$ we have that $\mathrm{supp}(\nu)=X$.
Therefore from now on we will assume that $\mathrm{supp}(p)=E$.

One of the advantages of working with the extension $\mathbb{L}_{f}$ is that we 
can compute explicitly its operator norm. In fact, 
\begin{align*}
	\| \mathbb{L}_{f} \|_{\mathrm{op}}
	&\equiv 
	\sup_{\|\varphi\|_{1}\leq 1 }\ 
	\int_{X} \left|\mathbb{L}_{f}\varphi \right| \, d\nu 
	\\
	&\leq 
	\sup_{\|\varphi\|_{1}\leq 1 }\ 
	\int_{X} \mathbb{L}_{f}\varphi^{+}+\mathbb{L}_{f}\varphi^{-} \, d\nu 
	\\
	&\leq 
	\sup_{\|\varphi\|_{1}\leq 1 }\ 
	\rho(\mathscr{L}_{f})\int_{X} \varphi^{+}+\varphi^{-} \, d\nu 
	\\
	&\leq 
	\rho(\mathscr{L}_{f})
\end{align*}
and the supremum is attained if we take the test function $\varphi\equiv 1$.

Note that for any $\varphi\in L^{1}(\nu)$ we have 
\[
\int_{X} \varphi\  d[\mathscr{L}_{f}^{*}\nu]
=
\int_{X} \mathbb{L}_{f} \varphi\  d\nu.
\]

In the sequel we consider $\mathbb{L}_{f}$, the natural bounded 
extension of the Ruelle operator
acting on $L^1(\nu)$, and its bounded extension $\mathbb{L}_{f}^{**}$ 
to the bidual of $L^1(\nu)$.  

The double transpose of $\mathbb{L}_{f}^{**}:L^1(\nu)^{**}\to L^1(\nu)^{**}$ 
sends $\xi\in L^1(\nu)^{**}$ to $\mathbb{L}_{f}^{**}\xi$, 
which is defined for each functional $\ell\in L^{1}(\nu)^{*}$ by
\[
\mathbb{L}_{f}^{**}\xi(\ell)  \equiv  \xi(\mathbb{L}_{f}^{*}\ell). 
\]

Next, we prove that $\mathbb{L}^{**}_{f}$ always   
has a positive eigenvector associated to $\rho(\mathscr{L}_{f})$.
\begin{displaymath}
	\xymatrixcolsep{5pc}\xymatrix{
		(L^{1}(\nu)^{**},\|\cdot\|_{\mathrm{op}})  \ar[r]^{\mathbb{L}_{f}^{**}} 
		& (L^{1}(\nu)^{**},\|\cdot\|_{\mathrm{op}}) 
		\\
		(L^{1}(\nu)^{*},\|\cdot\|_{\mathrm{op}})   \ar[u]^{\text{duality}} 
		& \ar[l]_{\mathbb{L}_{f}^{*} }  (L^{1}(\nu)^{*},\|\cdot\|_{\mathrm{op}}) \ar[u]_{\text{duality}}
		\\
		(L^1(\nu),\|\cdot\|_{1}) \ar[u]^{\text{duality}}  \ar[r]^{\mathbb{L}_{f}} 
		&  (L^1(\nu),\|\cdot\|_{1}) \ar[u]_{\text{duality}}   
		\\
		(C(X),\|\cdot\|_{1}) \ar[u]^{\text{extension}}  \ar[r]^{\mathscr{L}_{f}} 
		&  (C(X),\|\cdot\|_{1}). \ar[u]_{\text{extension}}   }\\
\end{displaymath}

\bigskip 

Now we are ready to prove one of the main results of this paper. 

\bigskip 

\begin{theorem}\label{teo-autofunc-bidual}
	Let $X=E^{\mathbb{N}}$, where $E$ is a compact metric space,  $f:X\to\mathbb{R}$
	be a continuous potential, $\nu\in\mathscr{G}^{*}(f)$ 
	and $\mathbb{L}_{f}:L^{1}(\nu)\to L^{1}(\nu)$ be the natural 
	extension of the Ruelle operator. Then there exists a positive element $\xi_{f}\in L^{1}(\nu)^{**}$
	such that 
	\[
	\mathbb{L}_{f}^{**} \xi_{f}
	=
	\rho(\mathscr{L}_{f})\xi_{f}.
	\]
\end{theorem}

\begin{proof} For each $n\in\mathbb{N}$ we define 
	\[ 
	\xi_n \equiv \frac{1}{\lambda^n_{f}}J( \mathbb{L}^{n}_{f}(1) ),
	\] 
	here $J$ denotes the natural map from $L^{1}(\nu)$ to its bidual. 
	Since the mapping $J$ is an isometry and $\|\mathbb{L}_{f}^n(1)\|_{L^{1}(\nu)}\leq \rho(\mathscr{L}_{f})^n$,
	it follows that $\|\xi_n\|\leq 1$. We remark that the mapping 
	$L^{1}(\nu)\ni \varphi\mapsto \nu(\varphi)$ is a norm-one element of $L^1(\nu)^{*}$. 
	This linear functional will be denoted by $\ell$ and for any $n\in\mathbb{N}$ 
	we have $\xi_n(\ell)=1$. 
	
	As we did before, if we look at the sequence $(\xi_n)_{n\in\mathbb{N}}$ as a 
	topological net, 
	it follows from the Banach-Alaoglu theorem that this net 
	has at least one convergent subnet 
	$(\xi_{i(d)})_{d\in D}$. Let $\xi_{f}$ the limit of such a subnet. Clearly this
	functional is non-null, positive and $\xi_{f}(\nu)=1$. Note that 
	$\rho(\mathscr{L}_{f})\xi_{n+1}= \mathbb{L}_{f}^{**}\xi_n$ for all $n\in\mathbb{N}$. 
	By using the weak-$*$ to weak-$*$ continuity of $\mathbb{L}_{f}^{**}$ 
	and a compactness argument we have 
	\begin{align*}
		\mathbb{L}_{f}^{**} \xi_{f}
		=
		\rho(\mathscr{L}_{f})\widetilde{\xi}_{f},
	\end{align*}
	for some $\widetilde{\xi}_{f}\in L^{1}(\nu)^{**}$.
	Note that $\xi_{n+1}(\ell)=1$ for all $n\in\mathbb{N}$ so
	$\widetilde{\xi}_{f}(\ell)=1$.
	Since 
	$\langle \ell\rangle \oplus \mathrm{ker}(\xi_{f})
	= 
	L^{1}(\nu)^{*} 
	= 
	\langle \ell\rangle \oplus \mathrm{ker}(\widetilde{\xi_{f}}) $ 
	follows that $\mathrm{ker}(\xi_{f}) =\mathrm{ker}(\widetilde{\xi_{f}})$
	and therefore $\widetilde{\xi_{f}}$ is non-zero multiple of $\xi_{f}$.
	The image of theses functionals evaluated at $\ell$ coincide, 
	so they are equal, which allows us to conclude that 
	$\xi_{f}$ is an eigenvector of $\mathbb{L}_{f}^{**}$. 
	To complete the proof we observe that for all 
	positive functional $\ell\in L^1(\nu)^{*}$, we have $\xi_{f}(\ell)\geq 0$.
\end{proof}

\section{Invariant Measures and Integrable Eigenfunctions}
\label{sec-inv-measures}

The next result is an important application of the 
existence of $\xi_f$. Before presenting this result let us
introduce some more notation. 
For each $A\in\mathscr{B}(X)$ and $\nu \in \mathscr{M}_{1}(X)$ 
we define a non-negative measure in $\mathscr{B}(X)$ so that 
$B\mapsto \nu(A\cap B)$. This measure will be simply denoted by $1_{A}\nu$. 
It will be convenient to identify the measure $1_{A}\nu$ with the element of 
$L^{1}(\nu)^*$ given by $\varphi\mapsto \nu(1_A \varphi)$. 
Recall that this functional was denoted by $\ell_{A}$, 
but in what follows the notation $1_{A}\nu$ is more convenient.

Recall that a dual pair $\langle X,Y\rangle$ is a pair of vector
spaces $X$ and $Y$ with a bilinear
map $X\times Y\ni (x,y)\longmapsto \langle y,x\rangle\in\mathbb{R}$,
which is nondegenerate in the sense that 
$\forall x, \exists y\ \langle y,x\rangle \neq 0$
and $\forall y, \exists x\ \ \langle y,x\rangle \neq 0$.
Here we are interested in the case where $X=L^{1}(\nu)$,
$Y=L^{1}(\nu)^{*}$, and $\langle y,x\rangle= y(x)$.

\begin{theorem}\label{teo-medida-gibbs-invariante}
	Let $f$ be a continuous potential, $\nu\in \mathscr{G}^{*}(f)$
	and let $\xi_{f}$ be a eigenfunction of $\mathbb{L}_{f}^{**}$
	as constructed above. Then the set function $\mu$ given by 
	$
	\mathscr{B}(X) \ni  A
	\mapsto 
	\xi_{f}(1_{A}\nu)
	$
	is a non-negative additive shift-invariant set function. 
\end{theorem}
\begin{proof}
	We first show that $\mu(\sigma^{-1}(A))=\mu(A)$.
	
	By the definition of $\mu$ we have 
	\begin{align}
		\mu(\sigma^{-1}(A)) 
		&= 
		\xi_{f}((1_{A}\circ \sigma)\cdot\nu)
		= 
		\rho(\mathscr{L}_{f})^{-1}
		\mathbb{L}_{f}^{**}(\xi_{f})((1_{A}\circ \sigma)\cdot\nu)
		\nonumber\\
		&=\label{eq-aux-constr-medida-gibbs} 
		\rho(\mathscr{L}_{f})^{-1}
		\xi_{f}
		\Big( \mathbb{L}_{f}^{*}((1_{A}\circ \sigma)\cdot\nu) \Big).
	\end{align}
	
	From the definition of the transpose of the Ruelle operator,
	the Riesz $L^p$ duality theorem and 
	basic properties of the Ruelle operator, we have for any $\varphi\in L^{1}(\nu)$ that

	\begin{align*}
		\langle \mathbb{L}_{f}^{*}&((1_{A}\circ \sigma)\cdot\nu), \varphi \rangle 
		\\[0.2cm]
		&=
		\langle (1_{A}\circ \sigma)\cdot\nu, \mathbb{L}_{f}\varphi \rangle 
		=
		\int_{X} 
		(1_{A}\circ \sigma)
		\mathbb{L}_{f}\varphi
		\ \text{d}\nu
		\\[0.2cm]
		&=
		\rho(\mathscr{L}_{f})^{-1}
		\int_{X} 
		(1_{A}\circ \sigma)
		\mathbb{L}_{f}\varphi
		\ \text{d}[\mathscr{L}_{f}^{*}\nu]
		=
		\rho(\mathscr{L}_{f})^{-1}
		\int_{X} 
		\mathbb{L}_{f}
		\Big(
		(1_{A}\circ \sigma)
		\mathbb{L}_{f}\varphi
		\Big)
		\ \text{d}\nu
		\\[0.2cm]
		&=
		\rho(\mathscr{L}_{f})^{-1}
		\int_{X} 
		1_{A}
		\mathbb{L}_{f}
		\mathbb{L}_{f}\varphi
		\ \text{d}\nu
		=
		\rho(\mathscr{L}_{f})^{-1}
		\int_{X} 
		\mathbb{L}_{f}
		\mathbb{L}_{f}\varphi
		\ \text{d}[1_{A}\nu]
		\\[0.2cm]
		&=
		\rho(\mathscr{L}_{f})^{-1}
		\langle 1_{A}\cdot\nu, \mathbb{L}_{f}\mathbb{L}_{f}\varphi \rangle 
		=
		\rho(\mathscr{L}_{f})^{-1}
		\langle \mathbb{L}_{f}^{*}\mathbb{L}_{f}^{*}(1_{A}\cdot\nu), \varphi \rangle. 
	\end{align*}
	Therefore 
	$
	\mathbb{L}_{f}^{*}((1_{A}\circ \sigma)\cdot\nu)
	=
	\rho(\mathscr{L}_{f})^{-1}
	\mathbb{L}_{f}^{*}\mathbb{L}_{f}^{*}(1_{A}\nu).
	$
	Replacing this identity in \eqref{eq-aux-constr-medida-gibbs}, 
	using the definition of the transposed operator, and the fact 
	that $\xi_{f}$ is an eigenfunction 
	of $\mathbb{L}_{f}^{**}$ two times, we obtain 
	\begin{align*}
		\mu(\sigma^{-1}(A)) 
		&=
		\rho(\mathscr{L}_{f})^{-1}
		\xi_{f}
		\Big( \mathbb{L}_{f}^{*}((1_{A}\circ \sigma)\cdot\nu) \Big)
		= 
		\rho(\mathscr{L}_{f})^{-2}
		\xi_{f}
		\Big( \mathbb{L}_{f}^{*}\mathbb{L}_{f}^{*}(1_{A}\nu) \Big)
		\\[0.2cm]
		&=
		\rho(\mathscr{L}_{f})^{-2}
		\mathbb{L}_{f}^{**}(\xi_{f})
		\Big( \mathbb{L}_{f}^{*}(1_{A}\nu) \Big)
		= 
		\rho(\mathscr{L}_{f})^{-1}
		\xi_{f}
		\Big( \mathbb{L}_{f}^{*}(1_{A}\nu) \Big)
		\\[0.2cm]
		&= 
		\rho(\mathscr{L}_{f})^{-1}
		\mathbb{L}_{f}^{**}(\xi_{f})(1_{A}\nu)
		= 
		\xi_{f}(1_{A}\nu)
		\\[0.2cm]
		&= 
		\mu(A).
	\end{align*}
	
	It is obvious that $\mu$ is a positive finitely additive set function, thus 
	concluding the proof. 
\end{proof}

\begin{theorem}\label{teo-Rad-Nik-autofuncao}
	Let $f$ be a continuous potential and $\nu\in\mathscr{G}^{*}(f)$. Assume that 
	the measure $\mu$ induced by $\xi_{f}\in L^1(\nu)^{**}$, as in Theorem \ref{teo-medida-gibbs-invariante},
	is a probability measure (countably additive). 
	Then $\mu\ll \nu$ and the Radon-Nikodym  derivative $d\mu/d\nu$ is a eigenfunction associated to 
	$\rho(\mathscr{L}_{f})$
	\begin{align}\label{auto-func-L1-RN}
		\mathbb{L}_{f} \frac{d\mu}{d\nu} = \rho(\mathscr{L}_{f}) \frac{d\mu}{d\nu}.
	\end{align}
\end{theorem}

\begin{proof}
	Since $\mu(A)=\xi_{f}(1_{A}\nu)$, for each Borel set $A$ it follows 
	that $\mu\ll \nu$. The proof of \eqref{auto-func-L1-RN} is similar to the one  
	given in \cite{CLS20}. Since it is simple, 
	for the convenience of the reader, we repeat the arguments below.	 
	
	For any continuous function $\varphi$ we have
	\begin{eqnarray*}
		\int_{X} \varphi\, \mathbb{L}_{f}\left( \frac{d\mu}{d\nu} \right)\, d\nu
		&=&
		\int_{X} \mathbb{L}_{f}\left(\varphi\circ\sigma \cdot \frac{d\mu}{d\nu}  \right)\, d\nu
		\\
		&=&
		\mathscr{L}_{f}\int_{X}\varphi\circ\sigma\cdot \frac{d\mu}{d\nu}\, d\nu
		=
		\mathscr{L}_{f}\int_{X}\varphi\circ\sigma\cdot\, d\mu
		\\
		&=&
		\mathscr{L}_{f}\int_{X}\varphi\, d\mu
		=
		\mathscr{L}_{f}\int_{X}\varphi \cdot \frac{d\mu}{d\nu}\, d\nu.
	\end{eqnarray*} 
\end{proof}

The above theorem is one of the most important results of this paper. It ensures 
that the spectral radius of the 
Ruelle operator $\mathbb{L}_{f}$ is
always an eigenvalue, as long as the finitely additive set function $\mu$ provided by
Theorem \ref{teo-medida-gibbs-invariante} is a probability measure. 
We also remark that the above Radon-Nikodym derivative
is not in general a continuous function. This derivative is continuous when the potential 
$f$ has good regularity properties, but for some continuous potentials $f$ 
we know that there is no continuous eigenfunction for $\mathscr{L}_{f}$ 
associated to the spectral radius $\rho(\mathscr{L}_{f})$, see for example reference
\cite{JLMS:JLMS12031}.

These previous results show that the existence of the maximal eigenfunctions
for $\mathbb{L}_{f}$ can be established by showing that the set function 
$\mu(A)=\xi_{f}(1_{A}\nu)$
is countably additive. This problem can overcome if one proves that 
$\mu$ is a regular measure because of Alexandroff's theorem. 
Let us elaborate on this.  
We say that a Borel additive signed measure $\mu$ on a topological space $X=(X,\tau)$
is called regular,
if given $A\in \mathscr{B}(\tau)$ and $\varepsilon>0$, 
there exists a closed set $F\subset A$ and an open set $O\supset A$ 
such that for all Borel sets $C\subset O\setminus F$ we have 
$|\mu(C)|<\varepsilon$. The Alexandroff theorem ensures that if 
$\mu$ is a bounded regular complex-valued additive set function defined
on the Borel sigma-algebra of a compact topological space, then $\mu$ is 
countably additive. For a proof see \cite[III.5.13]{MR0117523}.

We remark that the regularity required by the Alexandroff theorem 
will in general not be satisfied. To show this we 
provide an example where $\mu$ 
can not be a probability measure. 	

Alternatively, we can study the Yosida-Hewitt decomposition \cite{MR0045194} of $\mu = \mu_{a}+\mu_{c}$.
The term $\mu_{a}$ corresponds to the purely finitely additive part and $\mu_{c}$ is the countably
additive part. This decomposition is constructed in such way that the pure finitely additive part
is characterized by the following property: if $\gamma$ is a non-negative countably additive 
measure such that $0\leqslant \gamma\leqslant \mu_{a}$, then $\gamma=0$. Note that $\mu_{c}\ll \nu$
and so by Radon-Nikodym theorem, there is some $h$ such that $d\mu= d\mu_{a}+h\nu$. Of course,
if $\mu_{c}=0$ then $h=0$. In what follows we present an example where $\mu_{c}=0$. 

\section{Manneville-Pomeau Maps and the Yosida-Hewitt Decomposition}\label{sec-Man-Pomeau}

Fix a real number $\alpha\geq 2$ and consider the dynamics given 
by a Manneville-Pomeau 
map $T\equiv T_{\alpha}:[0,1]\to [0,1]$ which is 
defined for each $x\in [0,1]$ by the following expression: 
$T(x)= x+x^{\alpha} \  (\mathrm{mod}\ 1)$.
%\begin{center}
%	\begin{tikzpicture}[
%	%Escala e Estilo
%	scale=1.2,
%	axis/.style={very thick},
%	important line/.style={thick},
%	dashed line/.style={dashed, very thick},
%	dashed line2/.style={dashed, thick},
%	every node/.style={color=black,}
%	]
%	% eixos
%	\draw[axis] (0,0)  -- (2.21,0) node(xline)[right] {};
%	\draw[axis] (0,0) -- (0,2.2) node(yline)[above] {};
%	% linhas pontilhadas
%	\draw[dashed line] (1.3,0) -- (1.3,2.2);
%	\draw[dashed line2] (2.21,0) -- (2.20,2.2);
%	%ramo 1
%	\begin{scope}
%	\shade[bottom color=white, top color=green]
%	(1.3,2.2) parabola bend (0,0.02) ($(1.3,2.2)+(0,-2.15)$);
%	\end{scope}
%	\draw[important line] (0,0) parabola (1.3,2.2);
%	%ramo 2
%	\begin{scope}
%	\shade[bottom color=white, top color=red]
%	(2.2,2.20) parabola bend (1.3,0.05) ($(2.2,2.2)+(0,-2.15)$);
%	\end{scope}
%	\draw[important line] (1.3,0) parabola (2.2,2.2);
%	%
%	% Pontos no grafico
%	\draw[very thick] (2.21,1.5pt) -- (2.21,-1.5pt) node[below] {$1$};
%	\draw[very thick] (1.5pt,2.2) -- (-1.5pt,2.2) node[left] {$1$};
%	\draw[very thick] (0,1.5pt) -- (0,0) node[below] {$0$};
%	\fill[black] (1.3,0) circle (1pt) node[below](1.3,1) {$x_0(\alpha)$};
%	
%	\end{tikzpicture}
%\end{center}
Note that $T$ induces a continuous map on the one-dimensional torus $\mathbb{T}$, 
which will be denoted by $T:\mathbb{T}\to\mathbb{T}$. 
For any $x\in\mathbb{T}$ we have that $T'(x)= 1+\alpha x^{\alpha-1}$, where $T'$ denotes
the left derivative. 
The cover $\{[0,x_0(\alpha)], [x_0(\alpha),1] \}$ of $\mathbb{T}$ 
defines a Markov partition and we can prove that 
the Manneville-Pomeau maps we are considering here 
are semi-conjugated to the full-shift on two-symbol sequences 
\begin{displaymath}
	\xymatrixcolsep{4pc}\xymatrix{
		\{0,1\}^{\mathbb{N}} \ar[r]^{\sigma} \ar[d]_{\pi_{\alpha}}& 
		\{0,1\}^{\mathbb{N}} \ar[d]^{\pi_{\alpha}}   
		\\
		\mathbb{T} \ar[r]^{T}&  \mathbb{T}}
\end{displaymath}
that is, there is a surjective (not injective) map $\pi_{\alpha}:\{0,1\}^{\mathbb{N}}\to\mathbb{T}$
so that  $\pi_{\alpha}\circ\sigma = T\circ\pi_{\alpha}$. The semi-conjugacy $\pi_{\alpha}$
is 1 to 1 except for countably many points.

Consider the potential $g:\mathbb{T}\to \mathbb{R}$ given by 
$
g(x) \equiv -\log |T'(x)|.
$
In this setting the Ruelle operator 
$\mathcal{L}_{g}:C(\mathbb{T},\mathbb{R})\to C(\mathbb{T},\mathbb{R})$
is defined as follows 
\[
\mathcal{L}_{g}(\varphi)(x)
=
\sum_{y\in T^{-1}(x)}
\exp(g(y))\varphi(y).
\]

One can prove that the Lebesgue measure $\lambda$ is an eigenmeasure
associated to the spectral radius of $\mathcal{L}_{g}$, which
is in this case equal to one.  
Now we consider the continuous potential $f\equiv -\log T'\circ \pi_{\alpha}$ 
defined on the symbol space $\{0,1\}^{\mathbb{N}}$. 
It well-known that $\rho(\mathscr{L}_{f})=1$, 
and $\nu\circ \pi_{\alpha}^{-1}=\lambda$,  where $\mathscr{L}_{f}^{*}\nu=\nu$.
If we assume that the set function $\mu(A)=\xi_{f}(1_{A}\nu)$ is regular,
then $\mu$ is a shift-invariant probability measure 
satisfying $\mu \ll \nu$. 
This implies the existence of a $T$-invariant probability measure 
$\rho=\mu\circ\pi_{\alpha}^{-1}$ which is absolutely continuous with respect to
the Lebesgue measure. But for $\alpha\geq 2$ 
this is a contradiction with some known results about the 
non-existence of such measures, see \cite{MR1695915,MR1149493,MR599464}.
Actually, if we consider the Yosida-Hewitt decomposition of $\mu = \mu_{a}+\mu_{c}$,
a small variation of the previous argument implies that $\mu_{c} = 0$. Because
the way $\mu$ is constructed we believe that $F_{\mu_{a}}|_{C(X)}= \delta_{0}$, 
which is an equilibrium state for this system. 

This example shows that the hypotheses in Theorem \ref{teo-Rad-Nik-autofuncao} can 
not be weakened in general, that is, for some continuous potential $f$ it is possible 
that the set function $A\longmapsto \xi_{f}(1_{A}\nu)$ is not countably additive, 
but only a finitely additive measure.

Note that Alexandroff's Theorem \cite[III.5.13]{MR0117523} 
does not apply here because the inner and outer regularity can be broken. 
So in general the set function defined in Theorem \ref{teo-Rad-Nik-autofuncao} 
can really be only a finitely additive measure.

Note that our result implies a curious result about such Manneville-Pomeau 
maps. For these dynamical systems there is always 
at least one finitely additive $T$-invariant ``probability'' measure $\mu$
such that $\mu(Z)=0$, for all $Z\in\mathscr{B}([0,1])$ satisfying $\mathrm{Leb}(Z)=0$.  
It is not clear whether such objects can be used to 
get quantitative information on the asymptotic behavior of means of observables 
evaluated on typical orbits of the dynamical system. 

\medskip 

On the other hand, for $f$ in a dense subset of $C(X)$, there
is a positive eigenfunction $h_{f}\in C(X)$ for the Ruelle operator
$\mathscr{L}_{f}$ associated to $\rho(\mathscr{L}_{f})$. Therefore 
$\mathscr{L}_{f}^{**}J(h_f)= \rho(\mathscr{L}_{f})J(h_f)$, where $J$ is  
Jordan canonical mapping. By taking $\xi_f = J(h_f)$ one can see that 
\[
\int_{X}1_{A} h_f\, d\nu =\xi_{f}(1_{A}\nu),
\] 
thus proving that the regularity of the set function $A\longmapsto \xi_{f}(1_{A}\nu)$
is verified for potentials in a dense subset of $C(X)$.

\section{The Variational Problem}
\label{sec-prob-variacional}

Our next concern will be solving a generalization of the variational problem 
\eqref{variational-problem} using the theory developed here. To guide 
the discussion, we start by recalling the common strategy normally employed to
solve this problem when the state space is finite  
and the potential $f$ is H\"older. After
this discussion we turn our attention to general state spaces and 
continuous potentials.

Let us assume for a moment that $E=\{0,1,\ldots,d-1\}$ and 
$X=E^{\mathbb{N}}$. 
We recall that a probability measure 
$\mu\in\mathscr{M}_{\sigma}(X)$ 
is said to be an {\it equilibrium state}
for the potential $f$ if the supremum in \eqref{variational-problem} is attained
at $\mu$, i.e., 
\begin{align}\label{def-eq-state}
	h_{\mu}(\sigma) + \int_{X} f\, \text{d}\mu
	=
	\sup_{\nu\in \mathscr{M}_{\sigma}(X)}
	\{h_{\mu}(\sigma) + \int_{X} f\, \text{d}\nu\}.
\end{align}

In this finite-state space it is very well known 
how to construct an equilibrium state $\mu$ by means of the maximal eigendata of the Ruelle
operator. For example, if $f$ is a H\"older potential, then 
the classical Ruelle-Perron-Frobenius theorem implies the existence of 
$h_{f}$ and $\nu_{f}$, the maximal eigenfunction and eigenmeasure, 
of the Ruelle operator and its transpose, respectively. By taking a suitable
normalization one can prove that the probability measure $h_f\nu_{f}$ is 
the unique equilibrium state for $f$, see \cite{MR1085356}.

When working with possible uncountable state space $E$ 
and dynamics given by the left shift mapping, to avoid trivialities 
in the variational problem one needs to avoid using the Kolmogorov-Sinai entropy
in its formulation. To obtain a generalization of the finite-state 
space we adopt here the setting usually considered in Statistical Mechanics, 
see \cite{MR2807681,MR517873,MR1241537}.

If $\mu$ and $\nu$ are two arbitrary finite measures over $X$ and
$\mathscr{A}$ is a sub-$\sigma$-algebra of $\mathscr{B}(X)$ we 
define 
\[
\mathscr{H}_{\mathscr{A}}(\mu|\nu)
=
\begin{cases}
	\displaystyle\int_{X} 
	\frac{d\mu|_{\mathscr{A}}}{d\nu|_{\mathscr{A}}} 
	\log  \left(\frac{d\mu|_{\mathscr{A}}}{d\nu|_{\mathscr{A}}} \right)
	\, \text{d}\nu, &\ \text{if}\ \mu\ll\nu \ \text{on}\ \mathscr{A};
	\\[0.5cm]
	\infty,&\ \text{otherwise}.
\end{cases}
\]
The extended real number 
$\mathscr{H}_{\mathscr{A}}(\mu|\nu)$ is the negative of the {\it relative entropy} of $\mu$
with respect to $\nu$ on $\mathscr{A}$. 
Consider the product measure $\pmb{p}=\prod_{n\in\mathbb{N}} p$, 
where the probability measure $p$ still is the a priori measure 
used to construct the Ruelle operator. 
For each $\mu\in \mathscr{M}_{\sigma}(X)$ 
it is proved in \cite[Theorem 15.12]{MR2807681} that the following limits exist
\[
\mathtt{h}(\mu)\equiv 
-\lim_{n\to\infty} \frac{1}{n} \mathscr{H}_{\mathscr{B}(X)_n}(\mu|\pmb{p}),
\]
where $\mathscr{B}(X)_n$ is the $\sigma$-algebra generated by the 
projections $\{\pi_j:X\to E: 1\leq j\leq n\}$.
If $\#E=d$, we can show that $\mathtt{h}(\mu)+\log d=h_{\mu}(\sigma)$. Therefore
both entropies determine the same set of equilibrium states in the finite-alphabet setting.

As usual, we define the pressure of $f\in C(X)$ by 
\[
P(f)
\equiv 
\sup_{\mu\in \mathscr{M}_{\sigma}(X)}
\{\mathtt{h}(\mu) + \int_{X} f\, \text{d}\mu\}.
\]
If $f$ is an arbitrary continuous potential and 
$\rho(\mathscr{L}_{f})$ 
is the spectral radius of $\mathscr{L}_{f}$,
then we can show that $P(f)=\log\rho(\mathscr{L}_{f})$, 
see \cite{MR3377291} for H\"older potentials 
and \cite{CLS20} for continuous potentials.
If $f$ is a H\"older potential, it follows 
from one of the main results of \cite{MR3897924} that the probability measure 
$\mu$ given by 
$
\mathscr{B}(X) \ni  A
\mapsto 
\xi_{f}(1_{A}\nu)
$
has to be the unique equilibrium state for $f$. This fact follows 
from Theorem \ref{teo-medida-gibbs-invariante}, Theorem \ref{teo-Rad-Nik-autofuncao}
and the uniqueness of the equilibrium states for H\"older potentials.
If $f$ is a more general continuous potential we do not know what are necessary and sufficient conditions 
on the potential ensuring that the probability measure $\mu$ 
still is an equilibrium state for $f$.  
If $\xi_f$ does not belong to the image of the natural map, 
the set function  $\mu$  may neither be a probability measure nor a quasilocal measure. 
In that case the equilibrium measure does not necessarily satisfy the DLR conditions, and, using the   
Israel-Bishop-Phelps Theorem, many types of phase transitions and pathologies can be shown to occur,
for more details see \cite{MR0154092,MR556902,MR0459458,MR753071}.
At last, if Yosida-Hewitt decomposition of $\mu$ is not trivial, meaning $d\mu = d\mu_{a}+hd\nu$,
for some non-trivial $h$, we believe that $hd\nu$, up to a normalization, is an equilibrium state for $f$,
even if $f$ has low regularity properties. 
This could be obtained by generalizing the Rokhlin formula to this context.

\section{Concluding Remarks}\label{sec-conc-remarks}
A positive operator $T:L^{\infty}(\nu)\to L^{\infty}(\nu)$ satisfying 
$T(1_{X})=1_{X}$ is called a Markov operator. 
We can show that $\mathbb{L}^{*}_{(f-\log\rho(\mathscr{L}_{f}))}$ is a Markov operator. 
And so we can use Theorem 18.4 in \cite{MR2378491} to obtain a similar result 
as our Theorem \ref{teo-autofunc-bidual}. 
The difference of this approach compared with ours is that it 
uses the Brouwer-Schauder-Tychonoff fixed point theorem, while in our proof 
we used the Banach-Alaoglu theorem. 
Although both proofs are based on a compactness argument, 
our proof is more constructive because the eigenvector 
is obtained by means of an explicit sequence $(\xi_n)_{n\in\mathbb{N}}$.

As mentioned before, the maximal eigenfunctions of the Ruelle operator 
are useful in getting information on equilibrium states, ergodic optimization and 
large deviations. Here we study them in a weaker sense.  
We proved, under appropriate conditions, that the weak solutions (bidual solution) 
can be ``regularized'' to classic solutions (integrable functions). 
Of course, to exploit all the benefits of the previous results in Thermodynamic Formalism, 
one would like to take this regularization proceeding from 
\[
L^1(\nu)^{**} \to L^1(\nu) \to L^2(\nu) \to \ldots\to L^p(\nu) \to L^{\infty}(\nu) \to C(X).
\] 

To know what are all the continuous potentials for which the bidual solutions,
for the maximal eigenvalue problem, can be embedded in $C(X)$, 
beyond being an interesting mathematical problem,  would have 
several applications. 
In particular, some examples where this regularization is not 
possible from $L^1(\nu)^{**}$ to $C(X)$ 
involves potentials where we do have some sort of phase transition. 
It is not clear whether there is a connection between such obstructions and 
phase transition phenomena.  Of course, if $f$ is either a H\"older or Walters potential 
then it is possible to go all the way from the bidual space $L^1(\nu)^{**}$ to $C(X)$.

\section*{Acknowledgments}
L. Cioletti and R. Ruviaro were financed in part by the Coordena\c c\~ao de 
Aperfei\c coamento de Pessoal de N\'ivel Superior - Brasil (CAPES) - Finance Code 001.
L. Cioletti and R. Ruviaro acknowledges the CNPq through project PQ 310818/2015-0, 
DPP and FAP-DF for the financial support. 

We also thank Evgeny Verbitskiy for pointing out an error in a previous version of 
Theorem \ref{teo-Rad-Nik-autofuncao}
% Theorem B of 
%a previous version of this manuscript 
and for bringing to our attention 
the example in Section \ref{sec-Man-Pomeau}.
%\ref{sec-inv-measures}.

\bibliographystyle{plain}
\bibliography{references}

\end{document}